\documentclass[12pt]{amsart}
\usepackage{amscd,amsmath,amssymb,amsfonts,verbatim}
\usepackage[cmtip, all]{xy}
\usepackage{graphicx}
\usepackage[active]{srcltx}

\newtheorem{thm}{Theorem}[section]
\newtheorem{prop}[thm]{Proposition}
\newtheorem{lem}[thm]{Lemma}
\newtheorem{cor}[thm]{Corollary}

\theoremstyle{definition}

\newtheorem{defn}[thm]{Definition}
\newtheorem{remk}[thm]{Remark}
\newtheorem{remks}[thm]{Remarks}

\newtheorem{exm}[thm]{Example}
\newtheorem{exms}[thm]{Examples}
\newtheorem{notat}[thm]{Notation}

\numberwithin{equation}{section}

{\hfill$\square$\end{defn}}

{\hfill$\square$\end{remk}}

{\hfill$\square$\end{remks}}

{\hfill$\square$\end{exm}}

{\hfill$\square$\end{exms}}

{\hfill$\square$\end{notat}}

\newcommand{\sE}{{\mathcal E}}

\newcommand{\sL}{{\mathcal L}}

\newcommand{\sU}{{\mathcal U}}
\newcommand{\sV}{{\mathcal V}}

\newcommand{\A}{{\mathbb A}}

\newcommand{\C}{{\mathbb C}}

\newcommand{\F}{{\mathbb F}}
\newcommand{\G}{{\mathbb G}}

\newcommand{\bL}{{\mathbb L}}

\renewcommand{\P}{{\mathbb P}}
\newcommand{\Q}{{\mathbb Q}}

\newcommand{\Z}{{\mathbb Z}}

\newcommand{\surj}{\twoheadrightarrow}
\newcommand{\inj}{\hookrightarrow}

\newcommand{\Hom}{{\rm Hom}}

\newcommand{\del}{\partial}

\newcommand{\End}{{\operatorname{\text{End}}}}

\newcommand{\ds}{{/\kern-3pt/}}

\newcommand{\wh}{\widehat}

\renewcommand{\dim}{\text{\rm dim}}

\newcommand{\tuborg}{\left\{\begin{array}{ll}}
\newcommand{\sluttuborg}{\end{array}\right.}

\begin{document}
\title{Equivariant cobordism of flag varieties and of symmetric varieties}
\author{Valentina Kiritchenko}
\thanks{The first author was partially supported by the Dynasty Foundation fellowship and by grants: RFBR 10-01-00540-a, RFBR-CNRS 10-01-93110-a, AG Laboratory NRU-HSE, RF government grant, ag. 11.G34.31.0023, RF Federal Innovation Agency 02.740.11.0608, RF Ministry of Education and Science 16.740.11.0307}

\address{Faculty of Mathematics and Laboratory of Algebraic Geometry\\
Higher School of Economics\\
Vavilova St. 7, 112312 Moscow, Russia\\
and\\
Institute for Information Transmission Problems, Moscow}

\email{vkiritchenko@yahoo.ca}

\author{Amalendu Krishna}
\address{School of Mathematics, Tata Institute of Fundamental Research,
Homi Bhabha Road, Colaba, Mumbai, India}
\email{amal@math.tifr.res.in}

\baselineskip=10pt



\begin{abstract}
We obtain an explicit presentation for the equivariant cobordism ring of a complete
flag variety. An immediate corollary is a Borel presentation for the ordinary cobordism ring.
Another application is an equivariant Schubert calculus in cobordism.
We also describe the rational equivariant
cobordism rings of wonderful symmetric varieties of minimal rank.
\end{abstract}

\maketitle

\section{Introduction}\label{section:Intro}
Let $k$ be a field of characteristic zero, and $G$ a connected reductive group
split over $k$. Recall that a smooth {\em spherical variety} is a smooth $k$-scheme $X$ with an action of $G$
and a dense orbit of a Borel subgroup of $G$. Well-known examples of spherical varieties include flag varieties, toric varieties and wonderful compactifications of symmetric spaces. In this paper, we study the
equivariant cobordism rings of the following two classes of spherical varieties: the flag varieties and the wonderful symmetric varieties of minimal rank (the latter include wonderful compactifications of  semisimple groups of adjoint type).

The equivariant cohomology and the equivariant Chow groups of these two
classes of spherical varieties have been extensively studied before in
\cite{BDP}, \cite{LP}, \cite{Brion2}, \cite{Brion3}, and \cite{BJ}.
Based on the theory of algebraic cobordism by Levine and Morel \cite{LM},
and the construction of equivariant Chow groups by Totaro \cite{Totaro}
and Edidin-Graham \cite{EG}, the equivariant cobordism was
initially introduced in \cite{DP} for smooth varieties.
It was subsequently developed into a complete theory of equivariant
oriented Borel-Moore homology for all $k$-schemes in \cite{Krishna1}.
Similarly to equivariant cohomology, equivariant cobordism is a powerful tool for computing ordinary cobordism of the varieties with a group action.
The techniques of equivariant cobordism have been recently exploited to give
explicit descriptions of the ordinary cobordism rings of smooth toric varieties
in \cite{KU}, and that of the flag bundles over smooth schemes in
\cite{Krishna4}

In this paper, we give an explicit description of the equivariant
cobordism ring of a complete flag variety.
The ordinary cobordism rings of such varieties have been recently
described by Hornbostel--Kiritchenko \cite{HK} and Calm{\`e}s--Petrov--Zainoulline \cite{CPZ}.
Let $B\subset G$ be a Borel subgroup containing a split maximal torus $T$.
In Theorem \ref{thm:ECF-MainT}, we obtain an explicit presentation for $\Omega^*_T(G/B)$ tensored with $\Z[t_G^{-1}]$, where $t_G$ is the torsion index of $G$ (see Section \ref{section:MAINGB} for a definition).
As a consequence, one immediately obtains an expression for the ordinary cobordism rings of complete flag varieties (tensored with $\Z[t_G^{-1}]$) using a simple relation between the equivariant and the ordinary cobordism ({\sl cf.} \cite[Theorem~3.4]{Krishna2}). We also outline an equivariant Schubert calculus in
$\Omega^*_T(G/B)$ (see Subsection \ref{section:Dem}).

To compute $\Omega^*_T(G/B)$, we first prove the comparison theorems which relate the equivariant algebraic and complex cobordism rings of cellular varieties (see Section~\ref{section:ACC}) and then compute the equivariant complex cobordism $MU^{2*}_T(G/B)$ (see Section \ref{section:MAINGB}). The highlight of our proof
is that it only uses elementary techniques of equivariant geometry and
does not use any computation of the ordinary cobordism or cohomology.

In Section~\ref{section:CWSV}, we describe the rational $T$-equivariant
cobordism rings of wonderful symmetric varieties of minimal rank.
Again, this implies a description for their ordinary cobordism rings.
In particular, one gets a presentation for the cobordism ring of the wonderful compactification of an adjoint semisimple group.
The main ingredient of the proof is the localization theorem
for the equivariant cobordism rings for torus action
\cite[Theorem~7.8]{Krishna2}.
Once we have this tool, the final result
is obtained by adapting the argument of Brion-Joshua \cite{BJ} who obtained
an analogous presentation for the equivariant Chow ring.
As it turns out, similar steps
can be followed to compute the equivariant cobordism ring of any regular
compactification of a symmetric space of minimal rank.

\section{Recollection of equivariant cobordism}
\label{section:EAC}
In this section, we recollect the basic definitions and
properties of equivariant cobordism that we shall need in the sequel. For more details see
\cite{Krishna1}.
Let $k$ be a field of characteristic zero and let $G$ be a connected linear
algebraic group over $k$.

Let $\sV_k$ denote the category of quasi-projective $k$-schemes and
let $\sV^S_k$ denote the full subcategory of smooth quasi-projective
$k$-schemes. We denote the category of quasi-projective $k$-schemes with
linear $G$-action and $G$-equivariant maps by $\sV_G$ and the corresponding
subcategory of smooth schemes will be denoted by $\sV^S_G$.
In this text, a {\sl scheme} will always mean an object of $\sV_k$
and a $G$-scheme will mean an object of $\sV_G$. For all the definitions and
properties of algebraic cobordism that are used in this paper, we refer
the reader to \cite{LM}. All representations of $G$ will be finite-dimensional.
Let $\bL$ denote the Lazard ring which is same as the cobordism ring
$\Omega^*(k)$.

Recall the notion of a {\sl good pair}. For integer $j \ge 0$, let $V_j$ be a $G$-representation, and  $U_j \subset V_j$ an open subset such that the codimension of the complement is at least $j$. The pair  $\left(V_j,U_j\right)$ is called a {\sl good pair} corresponding to $j$ for the $G$-action if $G$ acts freely on $U_j$ and the quotient ${U_j}/G$ is a quasi-projective scheme.
Quotients ${U_j}/G$ approximate algebraically the {\em classifying space} $B_G$ (which is not algebraic) while $U_j$ approximate the {\em universal space} $E_G$.
It is known that such good pairs always exist.

Let $X$ be a smooth $G$-scheme.
For each $j \ge 0$, choose  a good pair $(V_j, U_j)$
corresponding to $j$.
For $i \in \Z$, set
\begin{equation}\label{eqn:E-cob*}
{\Omega_G^i(X)}_j =  \frac{\Omega^{i}\left({X\stackrel{G} {\times} U_j}\right)}
{F^{j}\Omega^{i}\left({X\stackrel{G} {\times} U_j}\right)}.
\end{equation}
Then it is known (\cite[Lemma~4.2, Remark~4.6]{Krishna1}) that
${\Omega_G^i(X)}_j$ is independent of the
choice of the good pair $(V_j, U_j)$. Moreover, there is a natural surjective
map $\Omega_G^i(X)_{j'} \surj \Omega_G^i(X)_j$ for $j' \ge j \ge 0$.
Here, $F^{\bullet}\Omega^*(X)$ is the coniveau filtration on $\Omega^*(X)$, i.e.
$F^j\Omega^*(X)$ is the set of all cobordism cycles $x \in \Omega^*(X)$
such that $x$ dies in $\Omega^*(X \setminus Y)$, where $Y \subset X$ is closed
of codimension at least $j$ ({\sl cf.} \cite[Section~3]{DP}).

\begin{defn}\label{defn:ECob}
Let $X$ be a smooth $k$-scheme with a $G$-action. For any
$i \in \Z$, we define the {\sl equivariant algebraic cobordism} of $X$ to be
\begin{equation}\label{eqn:Csmooth1}
\Omega^i_G(X) = {\underset {j} \varprojlim} \
{\Omega_G^i(X)}_j.
\end{equation}
\end{defn}

The reader should note from the above definition that unlike the ordinary
cobordism, the equivariant algebraic cobordism $\Omega_G^i(X)$ can be
non-zero for any $i \in \Z$. We set
\[
\Omega_G^*(X) = {\underset{i \in \Z} \bigoplus} \ \Omega_G^i(X).
\]
It is known that if $G$ is trivial, then the $G$-equivariant cobordism
reduces to the ordinary one.

\begin{remk}\label{remk:Csing}
If $X$ is a $G$-scheme of dimension $d$, which is not necessarily smooth, one defines the
equivariant cobordism of $X$ by
\begin{equation}
{\Omega^G_i(X)}_j =  {\underset {j} \varprojlim} \
\frac{\Omega_{i+l_j-g}\left({X\stackrel{G} {\times} U_j}\right)}
{F_{d+l_j-g-j}\Omega_{i+l_j-g}\left({X\stackrel{G} {\times} U_j}\right)},
\end{equation}
where $g=\dim(G)$ and $l_j=\dim(U_j)$.
Here, $F_{\bullet}\Omega_*(X)$ is the niveau filtration on $\Omega_*(X)$ such
that $F_j\Omega_*(X)$ is the union of the images of the natural
$\bL$-linear maps $\Omega_*(Y) \to \Omega_*(X)$ where $Y \subset X$ is closed
of dimension at most $j$. It is known that if $X$ is smooth of dimension
$d$, then $\Omega^i_G(X) \cong \Omega^G_{d-i}(X)$.
Since we shall be dealing mostly
with the smooth schemes in this paper, we do not need this definition
of equivariant cobordism.
\end{remk}

It is known that $\Omega_G^*(X)$ satisfies all the properties of a
multiplicative oriented cohomology theory like the ordinary cobordism.
In particular, it has pull-backs, projective push-forward,
Chern class of equivariant bundles, exterior and internal products,
homotopy invariance and projection formula. We refer to
\cite[Theorem~5.4]{Krishna1} for further detail.

The $G$-equivariant cobordism group $\Omega^*_G(k)$ of the ground field $k$
is denoted by $\Omega^*(B_G)$ and is called the cobordism of the
{\sl classifying space} of $G$. We shall often write it as $S(G)$.
We also recall the following result which gives a simpler description
of the equivariant cobordism and which will be used throughout this paper.

\begin{thm}\label{thm:NO-Niveu}$($\cite[Theorem~6.1]{Krishna1}$)$
Let ${\{(V_j, U_j)\}}_{j \ge 0}$ be a sequence of
good pairs for the $G$-action such that \\
$(i)$ $V_{j+1} = V_j \oplus W_j$ as representations of $G$ with ${\rm dim}(W_j)
> 0$ and \\
$(ii)$ $U_j \oplus W_j \subset U_{j+1}$ as $G$-invariant open subsets. \\
Then for any smooth scheme $X$ as above and any $i \in \Z$,
\[
\Omega_G^i(X) \xrightarrow{\cong} {\underset{j}\varprojlim} \
\Omega^{i}\left(X \stackrel{G}{\times} U_j\right).
\]
Moreover, such a sequence ${\{(V_j, U_j)\}}_{j \ge 0}$ of good pairs always
exists.
\end{thm}
 For the rest of this text, a {\sl sequence of good pairs}
${\{(V_j, U_j)\}}_{j \ge 0}$ will always mean a sequence as in
Theorem~\ref{thm:NO-Niveu}.

\subsection{Equivariant cobordism of the variety of complete flags in $k^n$}
To illustrate the definition of equivariant cobordism, we now compute $\Omega^*_T (G/B)$ for $G=GL_n(k)$.
Note that we will use a different (and computationally less involved) approach in Section \ref{section:MAINGB} where we compute $\Omega^*_T (G/B)$ for all reductive groups $G$.

We identify the points of the complete flag variety $X=G/B$
with {\em complete flags} in $k^n$.
A {\em complete flag} $F$ is  a strictly increasing sequence of subspaces
$$F=\{\{0\}=V^0\subsetneq V^1\subsetneq V^2\subsetneq\ldots\subsetneq V^n=k^n\}$$ with $\dim(V^k)=k$.
There are $n$ natural line bundles $\sL_1$,\ldots, $\sL_n$ on $X$, that is,
the fiber of $\sL_i$ at the point $F$ is equal to $V^{i}/V^{i-1}$.
These bundles are equivariant with respect to the left action of  the diagonal torus
$T\subset GL_n(k)$ on $X$, namely, $\sL_i$ corresponds to the character $\chi_i$ of $T$
given by the $i$-th entry of $T$.
For each $i=1,\ldots,n$, consider also the $T$-equivariant line bundle $L_i$ on ${\rm Spec}(k)$ corresponding to the character $\chi_i$.
In what follows, $\bL[[x_1,\ldots,x_n;t_1,\ldots,t_n]]$ denotes the {\bf graded} power
series ring in $x_1$,\ldots, $x_n$ and $t_1$,\ldots, $t_n$. Recall that for a graded ring $R$,
the {\em graded} power series ring $R[[x_1,\ldots,x_n]]$ consists of all finite linear combinations
of homogeneous (with respect to the total grading) power series (e.g., if $R$ has no terms of negative
degree then $R[[x_1,\ldots,x_n]]$ is just a ring of polynomials).

\begin{thm} \label{thm:GL_n} There is the following isomorphism
$$\Omega^*_T(X)\simeq\bL[[x_1,\ldots,x_n;t_1,\ldots,t_n]]/(s_i(x_1,\ldots,x_n)-s_i(t_1,\ldots,t_n), i=1,\ldots,n),$$
where $s_i(x_1,\ldots,x_n)$ denotes the $i$-th elementary symmetric function of the
variables $x_1$,\ldots, $x_n$.
The isomorphism sends $x_i$ and $t_i$, respectively, to the first $T$-equivariant Chern
classes $c_1^T(\sL_i)$ and $c_1^T(L_i)$.
\end{thm}
\begin{proof}
First, note that $\Omega^*_T(X)=\Omega^*_B(X)$ by \cite[Proposition 8.1]{Krishna1}, where $B$ is a Borel subgroup in $G$
(we choose $B$ to be the subgroup of the upper-triangular matrices).
For $N>n$, we can approximate the classifying space $B_B$ by partial flag varieties
$\F_{N,n}:=\F(N-n, N-n+1,\ldots,N-1,N)$ consisting of all flags
$$F=\{V^{N-n}\subsetneq V^{N-n+1}\subsetneq\ldots\subsetneq V^{N-1}\subsetneq k^N\}.$$
We choose exactly this approximation because its cobordism ring is easier to compute
via projective bundle formula than the cobordism ring of the dual flag variety
$\F(1, 2,\ldots,n; N)$ (for cohomology rings, this difference does not show up since the
Chern classes of dual vector bundles are the same up to a sign for the additive
formal group law).
Approximate $E_B$ by the variety $E_N:=\Hom^\circ(k^N,k^n)$ of all
projections of $k^N$ onto $k^n$.
Note that $\{(\Hom(k^N,k^n),E_N)\}_{N\ge n}$ is a sequence of good pairs as in Theorem \ref{thm:NO-Niveu} for the action of $GL_n$.

Denote by $\sE$ the tautological quotient bundle of rank $n$ on $\F_{N,n}$
(i.e., the fiber of $\sE$ at the point $F$ is equal to $k^N/V^{N-n}$).
For the complete flag variety $X$, we have that $X\times^B E_N$ is the flag variety
$\F(\sE)$ relative to the bundle $\sE$, whose points can be identified with
complete flags in the fibers of $\sE$.
Hence, we can compute the cobordism ring of $X\times^B E_N$ by the formula for the
cobordism rings of relative flag varieties \cite[Theorem 2.6]{HK}.
We get
$$\Omega^*(X\times^B E_N)=\Omega^*(\F(\sE))\simeq \Omega^*(\F_{N,n})[x_1,...,x_n]/I,$$
where $I$ is the ideal generated by the relations
$s_k(x_1,..,x_n)=c_k(\sE)$ for $1 \leq k \leq n$.
The isomorphism sends $x_i$ to the first Chern class of the line bundle
$\sL_i\times^B E_N$ on $X\times^B E_N$.

By the repeated use of the projective bundle formula (as in the proof of
\cite[Theorem 2.6]{HK}) we get that
$$\Omega^*(\F_{N,n})\simeq\bL[t_1,\ldots,t_n]/(h_N(t_n),h_{N-1}(t_{n-1},t_{n}),
\ldots,h_{N-n+1}(t_1,\ldots,t_n)),$$
where $t_i$ is the first Chern class of the $i$-th tautological line bundle on $\F_{N,n}$
(whose fiber at the point $F$ is equal to $V^{N-i+1}/V^{N-i}$),
and $h_k(t_i,\ldots,t_n)$ denotes the sum of all monomials of degree $k$ in
$t_i$,\ldots,$t_n$.

It is easy to deduce from the Whitney sum formula that $c_k(\sE)=s_k(t_1,\ldots,t_n)$.
Passing to the limit we get that $\Omega^i_B(X):={\underset{N}\varprojlim} \
\Omega^i(X\times^B E_N)$ consists
of all homogeneous power series of degree $i$ in $t_1$,\ldots, $t_{n}$ and
$x_1$,\ldots, $x_n$  modulo the relations $s_k(x_1,\ldots,x_n)=s_k(t_1,\ldots,t_n)$ for $1 \leq k \leq n$.
Indeed, all relations between $t_1$, \ldots, $t_n$ in $\Omega^*(\F_{N,n})$ are in degree
greater than $i$ if $N>i+n-1$.
\end{proof}

\section{Algebraic and complex cobordism}\label{section:ACC}
In this section, we assume our ground field to be the field of complex numbers
$\C$. To describe the equivariant algebraic cobordism ring
of flag varieties we first describe the equivariant complex
cobordism and then use some comparison results between the algebraic and
complex cobordism. Our main goal in this section is to establish such
comparison theorems.

For a $\C$-scheme $X$, the term $H^*(X, A)$ will denote the
singular cohomology of the space $X(\C)$ with coefficients in an abelian group
$A$. We shall use the notation $MU^*(X, A)$ for the term
$MU^*(X) \otimes_{\Z} A$, where $MU^*(-)$ denotes the complex cobordism,
a generalized cohomology theory on the category of CW-complexes.

Recall from \cite[\S 2]{Panin} that $X \mapsto MU^*(X(\C))$ is an example of an
oriented cohomology theory on $\sV^S_{\C}$. In fact, it is the universal
oriented cohomology theory in the category of CW-complexes which is
multiplicative in the sense that it has exterior and internal products.
One knows that $X \mapsto H^*(X, \Z)$ is also an example of a multiplicative
oriented cohomology theory on $\sV^S_{\C}$.

\subsection{Equivariant complex cobordism}\label{subsection:ECC*}
Recall (\cite[Section~ 7]{Krishna1}) that if $G$ is a complex Lie group and
$X$ is a finite CW-complex with a $G$-action, then its Borel {\sl equivariant
complex cobordism} is defined as
\begin{equation}\label{eqn:ECompC-R}
MU^*_G(X) := MU^*\left(X \stackrel{G}{\times} E_G\right),
\end{equation}
where $E_G \to B_G$ is a universal principal $G$-bundle and it is known that
$MU^*_G(X)$ is independent of the choice of this universal bundle.

\begin{defn}\label{defn:cobGP}
Let $\sU = \{(V_j, U_j)\}_{j \ge 0}$ be a sequence of good pairs for $G$-action.
For a linear algebraic group $G$ acting on
a $\C$-scheme $X$ and for any $i \in \Z$, we define
\begin{equation}\label{eqn:EQU}
MU^i_G\left(X, \sU\right) : = {\underset{j \ge 0}\varprojlim} \
MU^i\left(X \stackrel{G}{\times}U_j\right)
\end{equation}
and set $MU^*_G\left(X, \sU\right) = {\underset{i \in \Z} \bigoplus} \
MU^i_G\left(X, \sU\right)$.
We also set
\begin{equation}\label{eqn:EQU*1}
\Omega_G^i\left(X, \sU\right) : = {\underset{j \ge 0}\varprojlim} \
\Omega^{i}\left(X \stackrel{G}{\times}U_j\right) \ {\rm and} \
\Omega^*\left(X, \sU\right) = {\underset{i \in \Z} \bigoplus} \
\Omega_G^i\left(X, \sU\right).
\end{equation}
\end{defn}

It is easy to check as in \cite[Theorem~5.4]{Krishna1} that
$MU^*_G(-, \sU)$ and $\Omega_G^*\left(-, \sU\right)$ have all the functorial
properties of the equivariant cobordism. In particular, both are
contravariant functors on $\sV^S_G$ and $\Omega_G^*\left(-, \sU\right)$ is
also covariant for  projective maps. Moreover, the pull-back and the
push-forward maps commute with each other in a fiber diagram of smooth and
projective morphisms.

\begin{lem}\label{lem:Cob-lim}
Let $\sU = \{(V_j, U_j)\}_{j \ge 1}$ be a sequence of good pairs for the
$G$-action and  let $X$ be a smooth $G$-scheme such that
$H^*_G(X, \Z)$ is torsion-free.
There is an isomorphism
$MU^i_G(X) \to MU_G^i\left(X, \sU\right)$ of abelian groups for any
$i \in \Z$.
\end{lem}
\begin{proof}
Since $\sU$ is a sequence of good pairs for the $G$-action, the codimension
of the complement of $U_j$ in the $G$-representation $V_j$ is at least $j$.
In particular, the pair $(V_j, U_j)$ is $(j-1)$-connected. Taking the limit,
we see that $E_G = {\underset{j \ge 0}\bigcup} \ U_j$ is contractible and
hence $E_G \to {E_G}/G$ is the universal principal $G$-bundle and we can take
$B_G = {E_G}/G$. Since $X(\C)$ has the type of a finite CW-complex, we see
that $X_G = X \times^G E_G$ has a filtration by finite subcomplexes
\[
\emptyset = X_{-1} \subset X_0 \subset X_1 \subset \cdots \subset X_i \subset
\cdots \subset X_G
\]
with $X_j = X \times^G U_j$ and $X_G =
{\underset{j \ge 0}\bigcup} \ X_j$.
This yields the Milnor exact sequence
\begin{equation}\label{eqn:Alg-Comp2}
0 \to{\underset{j \ge 0}{\varprojlim}^1} \
MU^{i-1}\left(X_j\right) \to
MU^i_G\left(X\right)  \to
{\underset{j \ge 0}\varprojlim} \
MU^{i}\left(X_j\right) \to  0.
\end{equation}

Since $H^*_G\left(X, \Z\right) = H^*\left(X_G, \Z\right)$ is torsion-free,
it follows from \cite[Corollary~1]{Landweber} that first term in this exact
sequence is zero. This proves the lemma.
\end{proof}

\subsection{Comparison theorem}\label{subsection:CompT}
Recall from \cite[Example~1.9.1]{Fulton} that a scheme over a field $k$
(or an analytic space)
$L$ is called {\sl cellular} if it has a filtration
$\emptyset = L_{n+1} \subsetneq L_n \subsetneq \cdots \subsetneq L_1
\subsetneq L_0 = L$ by closed subschemes (subspaces) such that each $L_i
\setminus L_{i+1}$ is a disjoint union of affine spaces $\A^{r_i}_k$ ({\em cells}).
It follows from the Bruhat decomposition that varieties $G/B$ are cellular with cells labelled by elements of the Weyl group.
We begin with the following elementary and folklore result on cellular schemes.

\begin{lem}\label{lem:CELL}
Let $X$ be a $k$-scheme with a filtration
$\emptyset = X_{n+1} \subsetneq X_n \subsetneq \cdots \subsetneq X_1
\subsetneq X_0 = X$ by closed subschemes such that each $X_i \setminus X_{i+1}$
is a cellular scheme. Then $X$ is also a cellular scheme.
\end{lem}
\begin{proof}
It follows from our assumption that $X_n$ is cellular. It suffices to prove
by induction on the length of the filtration of $X$ that, if $Y \inj X$ is
a closed immersion of schemes such that $Y$ and $U = X \setminus Y$ are
cellular, then $X$ is also cellular. Consider the cellular decompositions
\[
\emptyset = Y_{l+1} \subsetneq Y_l \subsetneq \cdots \subsetneq Y_1
\subsetneq Y_0 = Y ,
\]
\[
\emptyset = U_{m+1} \subsetneq U_m \subsetneq \cdots \subsetneq U_1
\subsetneq U_0 = U
\]
of $Y$ and $U$. Set
 \[
X_i  =
\left\{ \begin{array}{ll}
Y \cup U_i & \mbox{if $0 \le i \le m+1$} \\
Y_{i-m-1} & \mbox{if $m+2 \le i \le m+l+2 \ .$}
\end{array}
\right .
\]
It is easy to verify that
$\{X_i\}_{0 \le i \le m+l+2}$ is a filtration of $X$ by closed
subschemes such that $X_i \setminus X_{i+1}$ is a disjoint union of affine spaces over $k$.
\end{proof}

Let $T$ be a torus of rank $n$ and let $\sU = \{(V_j, U_j)\}_{j \ge 1}$ be the
sequence of good pairs for $T$-action such that
each $(V_j, U_j) = {(V'_j, U'_j)}^{\oplus n}$, where $V'_j$ is the $j$-dimensional
representation of $\G_m$ with all weights $-1$ and $U'_j$ is the complement
of the origin and $T$ acts on $V_j$ diagonally.

\begin{defn}\label{defn:T-CELL}
A $\C$-scheme (or a scheme over any other field) $X$ with an action of $T$
is called {\sl $T$-equivariantly cellular}, if there is a filtration
$\emptyset = X_{n+1} \subsetneq X_n \subsetneq \cdots \subsetneq X_1
\subsetneq X_0 = X$ by $T$-invariant closed subschemes such that each
$X_i \setminus X_{i+1}$ is isomorphic to a disjoint union of representations $k^{r_i}$ of $T$.
\end{defn}
It follows from a theorem of Bialynicki-Birula \cite{BB}
(generalized to the case of non-algebraically closed fields by Hesselink
\cite{Hessel}) that if $X$ is a smooth projective variety with a $T$-action
such that the fixed point locus $X^T$ is isolated, then $X$ is $T$-equivariantly
cellular. In particular, a complete flag variety $G/B$ or, a smooth
projective toric variety is $T$-equivariantly cellular. It is obvious that
a $T$-equivariantly cellular scheme is cellular in the usual
sense.

\begin{prop}\label{prop:filter-Equiv}
Let $\sU = \{(V_j, U_j)\}_{j \ge 1}$ be as above, and $X$ a smooth
scheme with a $T$-action such that it is $T$-equivariantly cellular.
Then the natural map
\[
\Omega_T^*\left(X, \sU\right) \to MU_T^*\left(X, \sU\right)
\]
is an isomorphism.
\end{prop}
\begin{proof}
For any $\C$-scheme $Y$ with $T$ action, we set $Y^j =
Y \times^T U_j$ for $j \ge 1$.
Consider the $T$-equivariant cellular decomposition of $X$ as in
Definition~\ref{defn:T-CELL} and set $W_i = X_i \setminus X_{i+1}$.
It follows immediately that $X^j$ has a filtration
\[
\emptyset = (X^j)_{n+1} \subsetneq (X^j)_n \subsetneq \cdots \subsetneq (X^j)_1
\subsetneq (X^j)_0 = X^j,
\]
where $(X^j)_i = (X_i)^j =
X_i \times^T U_j$ and thus $(X^j)_i \setminus (X^j)_{i+1}
= (W_i)^j$.

Since ${U_j}/T \cong \left(\P^{j-1}\right)^n$ is cellular
and since
$(W_i)^j = W_i \times^T U_j \to {U_j}/T $ is a disjoint union of vector bundles, it follows that each $(X^j)_i = (W_i)^j$ is cellular.
We conclude from Lemma~\ref{lem:CELL} that $X^j$ is cellular.
In particular, the map $\Omega^*(X^j) \to MU^*(X^j)$ is an
isomorphism ({\sl cf.} \cite[Theorem~6.1]{HK}). The proposition now follows by
taking the limit over $j \ge 1$.
\end{proof}

\begin{lem}\label{lem:Tor-free}
Let $X$ be a $T$-equivariantly cellular scheme.
Then $H^*_T(X, \Z)$ is torsion-free.
\end{lem}
\begin{proof}
Let  $\sU = \{(V_j, U_j)\}_{j \ge 1}$ be a sequence of good pairs for $T$-action
as above. Since $H^i_T(X, \Z) \xrightarrow{\cong} H^i(X^j, \Z)$ for $j \gg 0$, it
suffices to show that $H^*(X^j, \Z)$ is torsion-free for any $j \ge 0$.
But we have shown in Proposition~\ref{prop:filter-Equiv} that each $X^j$
is cellular and hence $H^*(X^j, \Z)$ is a free abelian group.
\end{proof}

\begin{thm}\label{thm:Alg-Top}
Let $k$ be any field of characteristic zero and let $X$ be a smooth $k$-scheme
with an action of a split torus $T$. Assume that $X$ is $T$-equivariantly
cellular. Then there is a degree-doubling map
\[
\Phi^{\rm top}_X : \Omega^*_T(X) \to MU^{*}_T(X)
\]
which is a ring isomorphism
\end{thm}
\begin{proof}
If we fix a complex embedding $k \to \C$, then it follows from
our assumption and \cite[Theorem~4.7]{Krishna2} that
$\Omega^*_T(X) \cong S^{\oplus r} \cong \Omega^*_T(X_{\C})$, where $r$ is
the number of cells in $X$. Hence we can assume that our ground field is $\C$.

It follows from Lemma~\ref{lem:Tor-free} that
$H^*_T(X, \Z) = H^*\left(X \times^T E_G, \Z\right)$ is
torsion-free. We conclude from \cite[Proposition~7.4]{Krishna1} that there is
a ring homomorphism $\Phi^{\rm top}_X : \Omega^*_T(X) \to MU^{*}_T(X)$.

We now choose a sequence $\{(V_j, U_j)\}_{j \ge 1}$ of good pairs for the
$T$-action as in Proposition~\ref{prop:filter-Equiv}. It follows from
\cite[Theorem~6.1]{Krishna1} that for each $i \in \Z$,
$\Omega^i_T(X) \xrightarrow{\cong} \Omega_T^i\left(X, \sU\right)$,
and Lemma~\ref{lem:Cob-lim} implies that
$MU^{i}_T(X) \xrightarrow{\cong} MU^{i}_T\left(X, \sU\right)$. The theorem now
follows from Proposition~\ref{prop:filter-Equiv}.
\end{proof}

\begin{cor}\label{cor:flagT}
Let $G$ be a connected reductive group over $k$ and let $B$ be a Borel subgroup
containing a split maximal torus $T$. Then there is a ring isomorphism
\[
\Phi^{\rm top}_{G/B}: \Omega^*_T(G/B) \xrightarrow{\cong} MU^{*}_T(G/B).
\]
\end{cor}
\begin{proof} We have already commented above that $G/B$ is
$T$-equivariantly cellular. We now apply Theorem~\ref{thm:Alg-Top}.
\end{proof}

\section{Equivariant cobordism of $G/B$}\label{section:MAINGB}
 For the rest of the paper, $G$ denotes a split connected reductive group over $k$. We fix a split maximal torus $T$ of rank $n$ in $G$ and a Borel subgroup $B$ containing $T$. The Weyl group of $G$ is denoted by $W$.
In this section, we compute the equivariant cobordism ring $\Omega_T^*(G/B)$ of the complete flag variety $G/B$.

As we explained in the beginning of this text, to describe the $T$-equivariant
cobordism ring of the complete flag $G/B$, we do this first for the
complex cobordism and then use Corollary \ref{cor:flagT} to prove the analogous result in the algebraic
set-up. For the description of the equivariant complex cobordism, we need the
following special case of the Leray-Hirsch theorem
for a multiplicative generalized cohomology theory.

\begin{thm}[Leray-Hirsch]\label{thm:Top-LHT}
Let $X$ be a (possibly infinite) CW-complex with finite skeleta and
let $F \xrightarrow{i} E \xrightarrow{p} X$ be a fibration
such that the fiber $F$ is a finite CW-complex. Assume that
there are elements $\{e_1, \cdots , e_r\}$ in $MU^*(E)$ such that
$\{f_1 = i^*(e_1), \cdots , f_r = i^*(e_r)\}$ forms an $\bL$-basis of
$MU^*(F)$ for each fiber $F$ of the fibration. Assume furthermore that
$H^*(X, \Z)$ is torsion-free. Then the map
\begin{equation}\label{eqn:Top-LHT**1}
\Psi : MU^*(F) \otimes_{\bL} MU^*(X) \to MU^*(E)
\end{equation}
\[
\Psi\left({\underset{1 \le i \le r}\sum} \ f_i \otimes b_i\right)
= {\underset{1 \le i \le r}\sum} p^*(b_i) e_i
\]
is an isomorphism of $MU^*(X)$-modules. In particular, ${MU^*(E)}$
is a free ${MU^*(X)}$-module with the basis $\{e_1, \cdots , e_r\}$.
\end{thm}
\begin{proof}
This result is well known and can be found, for example, in
\cite[Theorem~15.47]{Switzer} and \cite[Theorem~3.1]{KonoT}.
We give a sketch of the main steps and in particular, explain where one needs
the fact that $H^*(X, \Z)$ is torsion-free.

The assignment $X \mapsto MU^*(X)$ is a multiplicative generalized
cohomology by \cite[Theorem~3.28]{KonoT}. Since this cohomology theory is
given by a spectrum, it satisfies the additivity axiom ({\sl cf.}
\cite[Chapter~2, \S 3]{KonoT}) by \cite[Theorem~2.21]{Prastaro}.
Hence we have the Atiyah-Hirzebruch spectral sequence
\begin{equation}\label{eqn:AHss}
E_2 = H^*(X, MU^*) \Rightarrow MU^*(X).
\end{equation}
The assumption of freeness and finite rank of $MU^*(F)$ over the ring
$MU^*$ implies that tensoring with $MU^*(F)$ is an exact functor
on the category of $MU^*$-modules.
In particular, the above spectral sequence becomes
\begin{equation}\label{eqn:AHss*}
E_2 = H^*(X, MU^*) \otimes_{MU^*} MU^*(F) \Rightarrow MU^*(X)
\otimes_{MU^*} MU^*(F).
\end{equation}
On the other hand, we also have the Serre spectral sequence
\begin{equation}\label{eqn:Sss}
E'_2 = H^*(X, MU^*(F)) \cong H^*(X, MU^*)\otimes_{MU^*} MU^*(F)
\Rightarrow MU^*(E).
\end{equation}

Applying the first spectral sequence and using
the assumption of the Leray-Hirsch theorem, we obtain a morphism
of the spectral sequences $E_2 \to E'_2$ which is clearly an
isomorphism ({\sl cf.} \cite[Theorem~15.47]{Switzer}).
Taking the limit of the two spectral sequences, we get the
desired isomorphism, provided we know that the two spectral sequences
converge strongly to $MU^*(E)$. Since the two spectral sequences
are isomorphic, we need to show that the any of the two converges.

On the other hand, it follows from the torsion-freeness of
$H^*(X, \Z)$ and \cite[Corollary~1]{Landweber} that
${\underset{n}{\varprojlim}^1} \ H^*(X_n, \Z) = 0$. The required convergence
of the Atiyah-Hirzebruch spectral sequence now follows from
\cite[Theorem~2.1]{BoJ}.
\end{proof}

\subsection{Equivariant complex cobordism of $G/B$}In what follows, we assume all spaces to be pointed and let $p_X : X \to{\rm pt}$ denote the structure map.
Let $MU^*(B_T) = MU^*_T({\rm pt})$ denote the coefficient ring of the
$T$-equivariant complex cobordism.
It is well known
(\cite{Landweber1}) that $MU^*(B_T)$ is isomorphic
to the graded power series $S=\bL[[t_1, \cdots , t_n]]$, where $t_i$ is the first Chern class of a $T$-equivariant line bundle on $B_T$ corresponding to the $i$-th basis character $\chi_i$ of $T$ (see \cite[Example~6.4]{Krishna1} for more details).
Note that each character $\chi$ of $T$  also gives rise
to the $B$-equivariant line bundle $\sL_{\chi}:=G/B\times^B L_{\chi}$ on $G/B$.
We will also use that $MU^*(B_T)=MU^*(B_B)$ is isomorphic
to $MU^*_G(G/B)$ since $G/B \times^G E_G =E_G/B$ and we can choose $E_G=E_B$.

For any finite CW-complex $X$ with a $G$-action, let $i_X : G/B \to
X \times^B E_G \cong
(X \times^B E_G) \times^G G/B
\xrightarrow{\pi_X} X \times^G E_G$ be the inclusion of the fiber
at the base point. Let $i : G/B \to {E_G}/B  \xrightarrow{\pi} B_G$ denote the
inclusion of the fiber when $X$ is the base point.
This gives rise to the following commutative diagram:
\begin{equation}\label{eqn:BBG}
\xymatrix@C.7pc{
MU^*(B_G) \ar[r]^{\pi^*} \ar[d]_{p^*_{G,X}} & MU^*(B_T) \ar[d]^{p^*_{T,X}} \ar[r]^{i^*}&
MU^*(G/B) \ar@{=}[d] \\
MU^*_G(X) \ar[r]_{\pi^*_X} & MU^*_T(X) \ar[r]_{i^*_X} & MU^*(G/B).}
\end{equation}

Recall that the {\em torsion index} of $G$ is defined as the smallest positive integer
$t_G$ such that $t_G$ times the class of a point in $H^{2d}(G/B,\Z)$ (where $d=\dim(G/B)$) belongs to the subring of $H^*(G/B,\Z)$ generated by the first Chern classes of line bundles $\sL_{\chi}$ (e.g., $t_G=1$ for $G=GL_n$, see \cite{Totaro2} for computations of $t_G$ for other groups).
If $G$ is simply connected then this subring is generated by $H^2(G/B,\Z)$.

For the rest of this section, an abelian group $A$ will actually mean
its extension $A \otimes_{\Z} R$, where $R = \Z[t^{-1}_G]$. In particular,
all the cohomology and the cobordism groups will be considered with
coefficients in $R$.

We shall use the following key fact to prove the main result of this section.

\begin{lem} \label{lem.surjective}
The homomorphism $i^*:MU^*_G(G/B)\to MU^*(G/B)$ is surjective over the ring $R$.
\end{lem}
\begin{proof} Since $MU^*_G(G/B)\simeq MU^*(B_T)\simeq S$, the image of $i^*$ is the subring of $MU^*(G/B)$ generated by the
first Chern classes of line bundles $\sL_{\chi}$.
To prove surjectivity of $i^*$ we have to show that $MU^*(G/B)$ is generated by the first
Chern classes.

Since $G/B$ is cellular the cobordism ring $MU^*(G/B)$ is a free
$\bL$-module. Choose a basis $\{e_w\}_{w\in W}$ in $MU^*(G/B)$ such that all $e_w$ are homogeneous
(e.g., take resolutions of the closures of cells).
Consider the homomorphism
$$\varphi:MU^*(G/B) \to MU^*(G/B)\otimes_{\bL}R.$$
Since $H^*(G/B,R)$ is torsion free, we have the isomorphism
$MU^*(G/B)\otimes_{\bL}R\simeq H^*(G/B, R)$.
Note that $H^*(G/B, R)$ is generated by the first Chern classes by definition of the torsion index, and the
homomorphism $\varphi$ takes the Chern classes to the Chern classes.
Hence, there exist homogeneous polynomials $\{\varrho_w\}_{w\in W}$, where
$\varrho_w\in R[t_1,\ldots,t_n]\subset S$ such that
$\varphi(e_w)=\varphi(i^*(\varrho_w)).$
Then the set of cobordism classes
$\{i^*(\varrho_w)\}_{w\in W}$ is a basis over $\bL$ in
$MU^*(G/B, R)$. Indeed, consider the transition matrix $A$ from the basis
$\{e_w\}_{w\in W}$ to this set (order $e_w$ and $\varrho_w$ so that their degrees
decrease).  The elements of $A$ are homogeneous elements of $\bL$ and
$A\otimes_{\bL}R$ is the identity matrix. By degree arguments, it follows that
the matrix $A$ is upper-triangular  and the diagonal elements are equal to
$1$, so $A$ is invertible.

Hence, $MU^*(G/B)$ has a basis consisting of polynomials in the first Chern
classes and the homomorphism $i^*$ is surjective over $R$.
\end{proof}

By Lemma \ref{lem.surjective}, we can choose polynomials $\{\varrho_w\}_{w\in W}$ in $MU^*_G(G/B) = S=\bL[[t_1,\ldots,t_n]]\simeq MU^*(B_T)$ such that
$\{i^*(\varrho_w)\}_{w\in W}$ form an $\bL$-basis in $MU^*(G/B)$. Set $\varrho_{w, X} =
p^*_{T,X}\left(\varrho_w\right)$ for each $w \in W$.
Define $\bL$-linear maps
\begin{equation}\label{eqn:BBG1}
s : MU^*(G/B) \to S, \
s_{X} : MU^*(G/B) \to MU^*_T(X)
\end{equation}
\[
s\left(i^*\left(\varrho_w\right)\right) = \varrho_w \ {\rm and} \
s_X\left(i^*\left(\varrho_w\right)\right) = \varrho_{w, X}.
\]
Note that maps $i_X$ and $i$ are $W$-equivariant. In particular,
the map $s$ is also $W$-equivariant.

\begin{lem}\label{lem:BBG-top}
Let $X$ be a finite CW-complex with a $G$-action such that $H^*_T(X, R)$ is
torsion-free. \\
$(i)$ The map $MU^*(G/B) \otimes_{\bL} MU^*_G(X) \to MU^*_T(X)$ which sends
$(i,x)$ to $s_X(b) \cdot \pi^*_X(x)$ is an isomorphism of
$MU^*_G(X)$-modules. In particular, $MU^*_T(X)$ is a free $MU^*_G(X)$-module
with the basis $\{\varrho_{w, X}\}_{w \in W}$. \\
$(ii)$ The map $S \times MU^*_G(X) \to MU^*_T(X)$ which sends
$(a, x)$ to $p^*_{T,X}(a) \cdot \pi^*_X(x)$ yields an isomorphism of
graded $\bL$-algebras
\begin{equation}\label{eqn:BBG-top1}
\Psi^{\rm top}_X : S \otimes_{MU^*(B_G)} MU^*_G(X) \xrightarrow{\cong} MU^*_T(X).
\end{equation}
\end{lem}
\begin{proof}
We first observe that we can use Lemma~\ref{lem:Cob-lim} to see that
$MU^*_G(X)$ and $MU^*_T(X)$ are $\bL$-algebras. Moreover, it follows from
our assumption and \cite[Proposition~2.1(i)]{HS} that $H^*_G(X, R)$ is
torsion-free.
Since $i^* = i^*_X \circ p^*_{T,X}$, we conclude from the above construction
that $i^*(\varrho_w) =  i^*_X \left(p^*_{T,X}(\varrho_w)\right) =
i^*_X\left(\varrho_{w, X}\right)$. Since
$\{i^*(\varrho_w)\}_{w\in W}$ form an $\bL$-basis of $MU^*(G/B)$
the first statement
now follows immediately by applying Theorem~\ref{thm:Top-LHT} to the
fiber bundle
$G/B \xrightarrow{i_X} X \times^B{E_G} \xrightarrow{\pi_X}
X \times^G E_G$. We have just observed that
$H^*(X\times^G E_G, R)$ is torsion-free.

To prove the second statement, we first notice that the map in
~\eqref{eqn:BBG-top1} is a morphism of $\bL$-algebras.
Moreover, it follows from the first part of the lemma that
$S \cong MU^*(B_T)$ is a free $MU^*(B_G)$-module with basis
$\{\varrho_w\}_{w \in W}$ and
$MU^*_T(X)$ is a free $MU^*_G(X)$-module with basis $\{\varrho_{w,X}\}_{w \in W}$.
In particular, $\Psi^{\rm top}_X$ takes the basis elements
$\varrho_w \otimes 1$ onto
the basis elements $\varrho_{w,X}$. Hence, it is an algebra isomorphism.
\end{proof}

We now compute $MU^*(B_G)$.

\begin{prop}\label{prop:Inv-top}
The natural map $MU^*(B_G) \to \left(MU^*(B_T)\right)^W$ is an
isomorphism of $R$-algebras.
\end{prop}
\begin{proof} 
Note that in the proof of Lemma \ref{lem.surjective}, we can choose 
$\varrho_{w_0}=1$ (here $w_0$ is the longest length element of the Weyl group).
Then applying Theorem~\ref{thm:Top-LHT}
to the fibration $G/B \xrightarrow{i} B_T \xrightarrow{\pi} B_G$ (as in the proof of
Lemma \ref{lem:BBG-top} for $X=pt$), we get
\begin{equation}\label{eqn:Inv-top*01}
\Psi(1 \otimes b) = \Psi(i^*(\varrho_{w_0}) \otimes b) = \pi^*(b)\varrho_{w_0}  = \pi^*(b) \
{\rm for \ any} \ b \in MU^*(B_G),
\end{equation}
where $\Psi$ is as in ~\eqref{eqn:Top-LHT**1}.
In particular, $\pi^*$ is the composite map
\begin{equation}\label{eqn:Inv-top*02}
\pi^* : MU^*(B_G) \xrightarrow{1 \otimes {id}} MU^*(G/B)\otimes_{\bL} MU^*(B_G)
\xrightarrow{\Psi} MU^*(B_T).
\end{equation}

Hence to prove the proposition, it suffices to
show using Theorem~\ref{thm:Top-LHT} that the map $1 \otimes {id}$
induces an isomorphism
$MU^*(B_G) \to \left(MU^*(G/B)\otimes_{\bL} MU^*(B_G)\right)^W$ over $R$.

We first show that the map
$MU^*(B_G) \xrightarrow{1 \otimes {id}} MU^*(G/B)\otimes_{\bL} MU^*(B_G)$
is split injective.
To do this, we only have to observe from the projection formula for the map
$p_{G/B} : G/B \to {\rm pt}$ that
${p_{G/B}}_* \left(\rho \cdot p^*_{G/B} (x)\right) =
{p_{G/B}}_*(\rho) \cdot x = x$, where $\rho \in MU^*(G/B)$
is the class of a point.
This gives a splitting of the map $p^*_{G/B}$ and hence a splitting of
$1 \otimes id=p^*_{G/B} \otimes id$.

To prove the surjectivity, we follow the proof of the analogous result for
the Chow groups in \cite[Theorem 1.3]{Totaro2}.
Since the Atiyah-Hirzebruch spectral sequence degenerates over the rationals
and since the analogue of our lemma is known for the singular cohomology
groups by \cite[Theorem~1.3(2)]{Totaro2}, we see that the proposition holds over the
rationals ({\sl cf.} \cite[Theorem~8.9]{Krishna1}).

We now let $\alpha : MU^*(G/B) \to \bL$ be the map $\alpha(y) =
{p_{G/B}}_*\left(\rho \cdot y\right)$ and set $\beta = \alpha \otimes id:
MU^*(G/B) \otimes_{\bL} MU^*(B_G) \to MU^*(B_G)$. Set
$f^* = p^*_{G/B} \otimes id$ and $f_* =  {p_{G/B}}_* \otimes id$.
The projection formula as above implies that $f^*\beta f^*(x) = f^*(x)$
for all  $x \in MU^*(B_G)$.  Thus $f^*\beta(y) = y$ for all
$y$ in the image of $1 \otimes id$. 
We identify $S \xrightarrow{\cong} MU^*(B_T)$ with
$MU^*(G/B) \otimes_{\bL} MU^*(B_G)$ over $R$ as in Lemma~\ref{lem:BBG-top} and
consider the commutative diagram
\begin{equation}\label{eqn:R-Q}
\xymatrix{
S \ar[d]_{g} \ar[r]^<<<<<<{\beta} & MU^*(B_G) \ar[d] 
\ar[r]^>>>>>{f^*} & S \ar[d]^{g} \\
S_{\Q} \ar[r]_<<<<<{\beta} & {MU^*(B_G)}_{\Q} \ar[r]_>>>>{f^*} & S_{\Q}}
\end{equation}
where $g:S \to S_{\Q}$ is the natural change of coefficients map.

Let us fix an element $x \in S^W$.
Since $g\left(S^W\right) \subseteq \left(S_{\Q}\right)^W$, it follows from our 
result over rationals that
\[
g\left(f^*\beta (x)\right) = f^*\beta \left(g(x)\right) = g(x).
\]
That is, $g\left(x - f^*\beta (x)\right) = 0$. Since $S$ is torsion-free,
we must have $x = f^*\beta (x)$ on the top row of ~\eqref{eqn:R-Q}. 
Since $x$ is an arbitrary element of
$S^W$, we conclude that $S^W \subseteq {\rm Image}(f^*)$ over $R$.

\end{proof}

\begin{remk}\label{remk:Alg-Winv}
We do not yet know if the map $S(G) \to S^W$ is
an isomorphism over $R$, although it is known to be true over the rationals
by \cite[Theorem~8.7]{Krishna1}.
\end{remk}

Combining Lemma~\ref{lem:BBG-top} and Proposition~\ref{prop:Inv-top},
we immediately get:

\begin{cor} \label{cor:iso} Let $X$ be a smooth $\C$-scheme with an action of $G$. Then
$$\Psi^{\rm top}_X : S\otimes_{S^W}MU^*_G(X)\xrightarrow{\simeq} MU^*_T(X).$$
In particular, $MU^*(G/B)$ is isomorphic to $S\otimes_{S^W}S$.
\end{cor}
This extends to cobordism a well-known result for cohomology (see e.g., \cite[Proposition 1(iii)]{Brion1}).

\subsection{Equivariant algebraic cobordism of $G/B$}\label{section:FIN1}
Using the natural map $r^G_T : \Omega^*_G(G/B) \to
\Omega^*_T(G/B)$ (\cite[Subsection 4.1]{Krishna1}) and the isomorphisms
(\cite[Propositions~5.5, 8.1]{Krishna1})
\[
S \cong \Omega^*_T(k) \cong \Omega^*_B(k) \cong \Omega^*_G(G/B),
\]
we get the {\sl characteristic} ring homomorphism
${\bf c}^{\rm eq}_{G/B} : S \to \Omega^*_T(G/B)$. We observe that since
${\bf c}^{\rm eq}_{G/B}$ is simply the change of group homomorphism, it is
the algebraic analogue of the restriction map $MU^*_G(G/B) \xrightarrow{\pi^*_X}
MU^*_T(G/B)$ in ~\eqref{eqn:BBG}. The structure map
$G/B \to {\rm {\rm Spec}}(k)$ gives the $\bL$-algebra map
$S \to \Omega^*_T(G/B)$, which is the algebraic analogue of the map
$p^*_{T, G/B}$ in ~\eqref{eqn:BBG}.

\begin{thm}\label{thm:ECF-MainT}
The natural map of $S$-algebras
\[
\Psi^{\rm alg}_{G/B} : S \otimes_{S^W} S \to \Omega^*_T(G/B)
\]
\[
\Psi^{\rm alg}_{G/B}( a \otimes b) = a \cdot {\bf c}^{\rm eq}_{G/B}(b)
\]
is an isomorphism over $R$.
\end{thm}
\begin{proof}

Using Corollary \ref{cor:iso}, we get a diagram
\begin{equation}\label{eqn:ECF-MainT1}
\xymatrix{
S \otimes_{S^W} S \ar[r]^{\Psi^{\rm alg}_{G/B}} \ar[dr]_{\Psi^{\rm top}_{G/B}} &
\Omega^*_T(G/B) \ar[d]^{\Phi^{\rm top}_{G/B}} \\
&MU^{*}_T(G/B)}
\end{equation}
which commutes by the above comparison of the various algebraic and topological
maps. The right vertical map is an
isomorphism by Corollary~\ref{cor:flagT} and the diagonal map is an
isomorphism by Corollary \ref{cor:iso}. We conclude that
$\Psi^{\rm alg}_{G/B}$ is an isomorphism too.
\end{proof}

Note that for $G=GL_n$, Theorem \ref{thm:ECF-MainT} reduces to Theorem \ref{thm:GL_n} since $R=\Z$ for $GL_n$.
However, the proof of Theorem \ref{thm:ECF-MainT} involves fewer computations and, in particular, does not rely on computation of ordinary cobordism rings. On the contrary, the ordinary cobordism ring can be easily recovered from Theorem \ref{thm:ECF-MainT}. The following result improves \cite[Theorem~8.1]{Krishna2} which was proven with the rational coefficients. The result below
also improves the computation of the non-equivariant cobordism ring of
$G/B$ in \cite[Theorem~13.12]{CPZ}, where a presentation of $\Omega^*(G/B)$
was obtained in terms of the completion of $S$ with respect to its augmentation ideal.
\begin{cor}\label{cor:OCFV}
There is an $R$-algebra isomorphism
\[
S \otimes_{S^W} \bL \xrightarrow{\cong} \Omega^*(G/B).
\]
\end{cor}
\begin{proof} This follows immediately from Theorem~\ref{thm:ECF-MainT}
and \cite[Theorem~3.4]{Krishna2}.
\end{proof}

\subsection{Divided difference operators} \label{section:Dem}
Various definitions of {\em generalized divided difference} (or {\em Demazure}) {\em operators} were given in \cite{BE} for complex cobordism and in \cite{HK,CPZ} for algebraic cobordism in order to establish Schubert calculus in $MU^*(G/B)$ and $\Omega^*(G/B)$.
Corollary \ref{cor:OCFV} allows us to compare these definitions.
We also outline Schubert calculus in equivariant cobordism using Theorem \ref{thm:ECF-MainT}.

Denote by $x_{\chi} \in S$ the first $T$-equivariant Chern class $c^T_1(L_{\chi})$ of the $T$-equivariant line bundle $L_{\chi}$ on ${\rm Spec}(k)$ associated with the character $\chi$ of $T$.
Recall that the isomorphism $S=\bL[[t_1,\ldots,t_n]]\simeq\Omega^*_T(k)$ sends $t_i$ to $x_{\chi_i}$ where $\chi_i$ is the $i$-th basis character of $T$.
The Weyl group $W_G$ acts on $S$: an element $w\in W_G$ sends $x_{\chi}$ to $x_{w\chi}$.
For each simple root $\alpha$, define an $\bL$-linear operator $\del_\alpha$ on the ring $S$:
$$\del_\alpha: f\mapsto (1+s_\alpha)\frac{f}{x_{-\alpha}},$$
where $s_\alpha\in W$ is the reflection corresponding to the root $\alpha$.
One can show that $\del_\alpha$ is indeed well-defined using arguments of \cite[Section 5]{HK} (in \cite{HK} the ring of all power series is considered but it is easy to check
that $\del_\alpha(f)$ is homogeneous if $f$ is homogeneous). It is also easy to check that $\del_\alpha$ is $S^W$-linear. In particular, $\del_\alpha$ descends to $S \otimes_{S^W} \bL$.

The comparison result below follows directly from definitions and Corollary \ref{cor:OCFV}.

\begin{enumerate}

\item Under the isomorphism $MU^*(B_T)\simeq S$, the operator $C_\alpha$ considered in
    \cite[Proposition 3]{BE} coincides with the operator $\del_\alpha$.

\item Under the isomorphism of $S \otimes_{S^W} \bL\simeq\Omega^*(G/B)$, the operator $\del_\alpha$ descends to the operator $A_\alpha$ defined in \cite[Section 3]{HK}.

\item The operator $\del_\alpha$ coincides with the restriction of the operator $C_\alpha$ from \cite[Definition 3.11]{CPZ} from the ring of all power series to $S$.
\end{enumerate}

Note that most of the operators considered above also have geometric meaning (see \cite{BE,HK,CPZ} for details). In particular, they were used to compute the {\em Bott-Samelson classes} in cobordism.

We now define an {\em equivariant generalized Demazure operator} $\del^T_\alpha$ on $S\otimes_{S^W}S$:
$$\del^T_\alpha: f\otimes g\mapsto \del_\alpha(f)\otimes g.$$
It is well-defined since $\del_\alpha$ is $S^W$-linear.
It follows immediately from Theorem~\ref{thm:ECF-MainT} that $\del^T_{\alpha}$
defines an $S$-linear operator on $\Omega^*_T(G/B)$.
Similarly to the ordinary cobordism, these operators can be used to compute the {\em equivariant Bott-Samelson classes}. We outline the main steps but omit those details that are the same as for the ordinary cobordism. We use notation and definitions of \cite{HK}.

Recall that to each sequence $I=\{\alpha_{1},\ldots,\alpha_{l}\}$ of simple roots, there corresponds a smooth {\em Bott-Samelson variety} $R_I$ endowed with an action of $B$ such that there is a $B$-equivariant map $R_I\to G/B$. In particular, each $R_I$ gives rise to the cobordism class $Z_I=[R_I\to G/B]$ as well as to the $T$-equivariant cobordism class $[Z_I]^T$. The latter can be expressed as follows.

\begin{thm}\label{thm:BS}$$[Z_I]^T=\del^T_{\alpha_l}\ldots\del^T_{\alpha_1}\left([pt]^T\right)$$
\end{thm}

The key ingredient is the following geometric interpretation of $\del^T_{\alpha}$.
Denote by $P_\alpha$ the minimal parabolic subgroup corresponding to the root $\alpha$.

\begin{lem}\label{lemma:BS} The operator $\del^T_\alpha$ is the composition of the change of group homomorphism $r^{P_\alpha}_T : \Omega^*_{P_\alpha}(G/B) \to \Omega^*_T(G/B)$ and the push-forward map
$r_{P_\alpha}^T : \Omega^*_T(G/B)\to \Omega^*_{P_\alpha}(G/B) $:
$$\del_\alpha=r^{P_\alpha}_T r_{P_\alpha}^T.$$
\end{lem}
Similarly to \cite[Corollary 2.3]{HK}, this lemma follows from the Vishik-Quillen formula
\cite[Proposition 2.1]{HK} applied to $\P^1$-fibrations
$G/B\times^T U_j\to G/B\times^{P_\alpha}U_j$ (for a sequence of good pairs $\{(V_j,U_j)\}$ for the action of $P_\alpha$). Note here that $r^T_{P_{\alpha}}$
is defined by taking the limit over the push-forward
maps on the ordinary cobordism groups corresponding to the projective
morphism $G/B\times^T U_j\to G/B\times^{P_\alpha}U_j$.
Theorem \ref{thm:BS} then can be deduced from Lemma \ref{lemma:BS}  by the same arguments as in \cite[Theorem 3.2]{HK}.

\section{Cobordism ring of wonderful symmetric varieties}
\label{section:CWSV}
The wonderful (or more generally, regular) compactifications of symmetric varieties form a large class
of spherical varieties. In fact, much of the study of a very large class of
spherical varieties can be reduced to the case of symmetric varieties
({\sl cf.} \cite{Res}). In this section, we compute the equivariant cobordism ring of the wonderful symmetric
varieties of minimal rank (see Theorem~\ref{thm:ESPH}). A presentation  for the equivariant cohomology of
the wonderful group compactification analogous to Theorem~\ref{thm:ESPH} below was obtained by Littelmann and
Procesi in \cite{LP} and the corresponding result for the equivariant
Chow ring was obtained by Brion in \cite[Theorem~3.1]{Brion3}.
This result of Brion was later generalized to the case of wonderful symmetric varieties of
minimal rank by Brion and Joshua in \cite[Theorem~2.2.1]{BJ}.

Our proof of Theorem~\ref{thm:ESPH} follows the strategy of \cite{BJ}.
The two new ingredients in our case are the localization theorem for torus action in cobordism
({\sl cf.} \cite[Theorem~7.9]{Krishna2}),
and a divisibility result (Lemma \ref{lem:div}) in the ring $S = \Omega^*_T(k)$.

\subsection{Symmetric varieties}\label{subsection:symV}
We now define symmetric varieties and describe their basic structural properties following \cite{BJ}.
For the rest of the paper, we assume that $G$ is of adjoint type. Denote by $\Sigma^+$ the set of positive roots of $G$
with respect to the Borel subgroup $B$. Let $\Delta_G = \{\alpha_1, \ldots, \alpha_n\}$ be the set
of positive simple roots which form a basis of the root system and
let $\{s_{\alpha_1}, \cdots , s_{\alpha_n}\}$ be the set of associated reflections.
Since $G$ is adjoint, $\Delta_G$ is also a basis of the character group
$\wh T$. Recall that $W = W_G$ denotes the Weyl group of $G$.

Let $\theta$ be an involutive automorphism of $G$ and let $K \subset G$ be
the subgroup of fixed points $G^{\theta}$. The homogeneous space $G/K$ is
called a {\sl symmetric space}. Let $K^0$ denote the identity component of $K$
and set $T_K = \left(T \cap K\right)^0$. It is then known
(\cite[Lemma~1.4.1]{BJ}) that $K^0$ is reductive and the roots of $(K^0, T_K)$
are exactly the restrictions to $T_K$ of the roots of $(G, T)$. Moreover,
the Weyl group of $(K^0, T_K)$ is identified with $W^{\theta}$.
Let $P$ be a minimal {\em $\theta$-split} parabolic subgroup of $G$ (a  parabolic subgroup $P$ is {\em $\theta$-split} if $\theta(P)$ is opposite to $P$),
and $L = P \cap \theta(P)$ a $\theta$-stable Levi subgroup of $P$.
Then every maximal torus of $L$ is also $\theta$-stable. We assume that $T$ is such a
torus so that $T = T^{\theta} T^{-\theta}$ and the identity component $A = T^{-\theta, 0}$ is a maximal
{\em $\theta$-split} subtorus of $G$ (a torus is {\em $\theta$-split} if $\theta$ acts on it via the inverse map $g\mapsto g^{-1}$). The rank of such a torus $A$ is called the
{\sl rank of the symmetric space} $G/K$. Since $T^{\theta}\cap T^{-\theta}$ is
finite, we get
\[
{\rm rk}(G) \le {\rm rk}(K) + {\rm rk}(G/K)
\]
and the equality holds if and only if $T^{\theta, 0}$ is a maximal torus
of $K^0$ and $T^{- \theta, 0}$ is a maximal $\theta$-split torus.
If this happens, one says that the symmetric space $G/K$ is of {\sl minimal rank}.

Let $\Sigma_L \subset \Sigma$ be the set of roots of $L$, and $\Delta_L
\subset \Delta_G$ the subset of simple roots of $L$. If $p : \wh{T}
\to \wh{A}$ denotes the restriction map, then its image is a reduced root system
denoted by $\Sigma_{G/K}$ and $\Delta_{G/K} : = p\left(\Delta_G \setminus
\Delta_L\right)$ is a basis of $\Sigma_{G/K}$. This set is also identified
with $\{\alpha - \theta(\alpha) | \alpha \in \Delta_G \setminus \Delta_L\}$
under the projection $p$. Moreover, there is an exact
sequence
\begin{equation}\label{eqn:sym1}
1 \to W_L \to W^{\theta} \xrightarrow{p} W_{G/K} \to 1.
\end{equation}
A representative of the reflection of $W_{G/K}$ associated to the
root $\alpha - \theta(\alpha) \in \Delta_{G/K}$ is $s_{\alpha}s_{\theta(\alpha)}$.

\begin{defn}\label{defn:WCSV}
Let $G/K$ be a symmetric space as above. The {\sl wonderful compactification}
of $G/K$ is a smooth and projective $G$-variety $X$
such that \\
$(i)$ There is an open orbit of $G$ in $X$ isomorphic to $G/K$. \\
$(ii)$ The complement of this open orbit is the union of $r =
{\rm rk}(G/K)$ smooth prime divisors $\{X_1, \cdots X_r\}$ with strict normal
crossings. \\
$(iii)$ The $G$-orbit closures in $X$ are precisely the various intersections
of the above prime divisors. In particular, all $G$-orbit closures are smooth.
\\
$(iv)$ The unique closed orbit $X_1 \cap \cdots \cap X_r$ is isomorphic to
$G/P$.
\end{defn}
We say that $X$ is a {\sl wonderful symmetric variety}. This is said to be
of minimal rank if $G/K$ is so. The existence of such compactifications of
symmetric spaces is known by the work of De Concini-Procesi \cite{DeCP}
and De Concini-Springer \cite{DeCS}. A well-known example of a
wonderful symmetric variety is the space of complete conics (which is not
of minimal rank).

Possibly, the simplest example of symmetric varieties of minimal rank is
when $G = {\bf G} \times {\bf G}$ where ${\bf G}$ is a semisimple group of
adjoint type, and $\theta$ interchanges the factors. In this case,
we have $K = {\rm diag}({\bf G})$ and $G/K \cong {\bf G}$, where $G$ acts by left and right multiplications.
Furthermore,
$T = {\bf T} \times {\bf T}$ where ${\bf T}$ is a maximal torus of ${\bf G}$.
Thus, $T_K = {\rm diag}({\bf T})$, $A = \{(x, x^{-1})| x \in {\bf T}\}$,
$L = T$ and $W_K = W_{G/K} = {\rm diag}(W_{\bf G}) \subset
W_{\bf G} \times W_{\bf G} = W$. In this case, the variety $X$ is called
the wonderful group compactification.
We refer to \cite[Example~1.4.4]{BJ} for an exhaustive list of
symmetric spaces of minimal rank.

Let $X$ be the wonderful compactification of a symmetric space $G/K$ of
minimal rank. Let $Y \subset X$ denote the closure of $T/{T_K}$ in $X$.
It is known that $Y$ is smooth and is the toric variety associated to the
Weyl chambers of the root datum $(G/K, \Sigma_{G/K})$. Let $z$ denote the
unique $T$-fixed point of the affine $T$-stable open subset $Y_0$ of $Y$ given
by the positive Weyl chamber of $\Sigma_{G/K}$.
It is well known that $X$ has an isolated set of fixed points
for the $T$-action. Moreover, it is also known by \cite[\S 10]{Tch} that
$X$ contains only finitely many $T$-stable curves. We shall need the following
description of the fixed points and $T$-stable curves.

\begin{lem}$($\cite[Lemma~2.1.1]{BJ}$)$\label{lem:stableC}
$(i)$ The $T$-stable points in $X$(resp. $Y$) are exactly the points $w\cdot z$,
where $w \in W$ (resp. $W_K$) and these fixed points are parameterized by
$W/{W_L}$ (resp. $W_{G/K}$). \\
$(ii)$ For any $\alpha \in \Sigma^+ \setminus \Sigma^+_L$, there exists a unique
irreducible $T$-stable curve $C_{z, \alpha}$ which contains $z$ and on which $T$
acts through the character $\alpha$. The $T$-fixed points in $C_{z, \alpha}$ are
$z$ and $s_{\alpha} \cdot z$. \\
$(iii)$ For any $\gamma = \alpha - \theta (\alpha) \in \Delta_{G/K}$, there
exists a unique irreducible $T$-stable curve $C_{z, \gamma}$ which contains $z$
and on which $T$ acts through its character $\gamma$. The $T$-fixed points
in $C_{z, \gamma}$ are exactly $z$ and $s_{\alpha}s_{\theta(\alpha)} \cdot z$. \\
$(iv)$ The irreducible $T$-stable curves in $X$ are the $W$-translates of the
curves $C_{z, \alpha}$ and $C_{z, \gamma}$. They are all isomorphic to $\P^1$. \\
$(v)$ The irreducible $T$-stable curves in $Y$ are the $W_{G/K}$-translates of
the curves $C_{z, \gamma}$.
\end{lem}

\subsection{Cobordism ring of symmetric varieties}
\label{subsection:CSV}
To prove our main result, we will also need the following result on
divisibility in the graded power series ring
$S=\bL[[t_1,\ldots,t_n]]$. We use notation of Subsection \ref{section:Dem}.

\begin{lem}\label{lem:div} For any $f\in S$ and any root $\alpha$, we have
\begin{equation}
f \equiv s_{\alpha}(f)  \ \ ({\rm mod} \ x_{\alpha}).
\end{equation}
\end{lem}
\begin{proof}It is enough to check this lemma for all monomials in $t_1$,\ldots, $t_n$.

First, check the case $f=t_i$. For each $\chi\in{\wh T}$ we have $s_\alpha\chi=\chi-(\chi,\alpha)\alpha$, where $(\chi,\alpha)$ is integer. Put $k=-(\chi,\alpha)$.  We can express $x_\chi-x_{s_\alpha\chi}=x_\chi-x_{\chi+k\alpha}$ as a formal power series $H(x,y)\in\bL[[x,y]]$ in $x=x_{\chi}$ and $y=x_{\alpha}$ using the universal formal group law. Then $H(x,y)$ is homogeneous and divisible by $y$ \cite[(2.5.1)]{LM} so that the ratio $\frac{H(x,y)}{y}$ is a homogeneous power series. In particular, $t_i-s_\alpha(t_i)$ is divisible by $x_\alpha$.

Next, note that if the lemma holds for $f$ and $g$, then it also holds for $fg$, since
$fg-s_\alpha(fg)=(f-s_\alpha(f))g+s_\alpha(f)(g-s_\alpha(g))$.
In particular, the lemma holds for any monomial in $t_1$,\ldots, $t_n$ as desired.
\end{proof}

\begin{thm}\label{thm:ESPH}
Let $X$ be a wonderful symmetric variety of minimal rank. Then the composite
map
\begin{equation}\label{eqn:ESPH1}
s^G_T : \Omega^*_G(X) \to \left(\Omega^*_T(X)\right)^W \to
\left(\Omega^*_T(X)\right)^{W_K} \to  \left(\Omega^*_T(Y)\right)^{W_K}
\end{equation}
is a ring isomorphism with the rational coefficients.
\end{thm}
\begin{proof}
All the arrows in ~\eqref{eqn:ESPH1} are canonical ring homomorphisms.
The isomorphism of the first arrow follows from \cite[Theorem~8.7]{Krishna1}.
Thus, it suffices to show that the map $\left(\Omega^*_T(X)\right)^W
\to  \left(\Omega^*_T(Y)\right)^{W_K}$ is an isomorphism.
We prove this by adapting the argument of
\cite[Theorem~2.2.1]{BJ}.

Since $X$ has only finitely many $T$-fixed points and finitely many $T$-stable
curves, it follows from \cite[Theorem~7.9]{Krishna2} and
Lemma~\ref{lem:stableC} that $\Omega^*_T(X)$ is isomorphic as an $S$-algebra
to the space of tuples ${\left(f_{w\cdot z}\right)}_{w \in W/{W_L}}$ of elements of
$S$ such that
\[
f_{v\cdot z} \equiv f_{w\cdot z}  \ \ ({\rm mod} \ x_{\chi})
\]
whenever $v \cdot z$ and $w \cdot z$ lie in an irreducible $T$-stable curve
on which $T$ acts through its character $\chi$.  Under this isomorphism, the ring $S$ is identified with the constant
tuples $(f)$.

We deduce from this that $\left(\Omega^*_T(X)\right)^W$ is isomorphic, via the
restriction to the $T$-fixed point $z$, to the subring of $S^{W_L}$ consisting
of those $f$ such that
\begin{equation}\label{eqn:ESPH2}
v^{-1}(f) \equiv w^{-1}(f)  \ \ ({\rm mod} \ x_{\chi})
\end{equation}
for all $v, w$ and $\chi$ as above. Using Lemma~\ref{lem:stableC}, we
conclude that $\left(\Omega^*_T(X)\right)^W$ is isomorphic to the subring
of $S^{W_L}$  consisting of those $f$ such that
\begin{equation}\label{eqn:ESPH3}
f \equiv s_{\alpha}(f)  \ \ ({\rm mod} \ x_{\alpha})
\end{equation}
for $\alpha \in \Sigma^+ \setminus \Sigma^+_L$ and those $f$ such that
\begin{equation}\label{eqn:ESPH4}
f \equiv s_{\alpha} s_{\theta(\alpha)}(f)  \ \
({\rm mod} \ x_{\gamma})
\end{equation}
for $\gamma = \alpha - \theta(\alpha) \in \Delta_{G/K}$.
However, it follows from Lemma~\ref{lem:div} that ~\eqref{eqn:ESPH3} holds
for all $f \in S$. We conclude from this that $\left(\Omega^*_T(X)\right)^W$
is isomorphic to the subring of $S^{W_L}$  consisting of those $f$ such that
~\eqref{eqn:ESPH4} holds for $\gamma = \alpha - \theta(\alpha) \in \Delta_{G/K}$.

Doing the similar calculation for $Y$ and using Lemma~\ref{lem:stableC} and
\cite[Theorem~7.9]{Krishna2} again, we see that
$\left(\Omega^*_T(Y)\right)^{W_K}$ is isomorphic to the same subring of $S$.
This completes the proof of the theorem.
\end{proof}

\begin{remk}\label{remk:toric}
Since $Y$ is a smooth toric variety, $\Omega^*_T(Y)$ can be explicitly
calculated in terms of generators and relations using \cite[Theorem~1.1]{KU}.
Combining this with Theorem~\ref{thm:ESPH}, one gets a simple way of
computing the equivariant cobordism ring of wonderful symmetric varieties
of minimal rank.
\end{remk}

\begin{exm}
If $G=PSL_2(k)\times PSL_2(k)$, and $\theta$ interchanges both factors then $G/K\simeq PSL_2(k)$ admits a
unique wonderful compactification $X=\P^3$. Namely, $\P^3$ can be regarded as $\P(\End(k^2))$, where $G$ acts by
left and right multiplications.
The toric variety $Y$ is $\P^1$ in this case. The torus $T\subset G$ is two-dimensional,
and $S=\bL[[t_1,t_2]]$. Both $\Omega^*_T(X)$ and $\Omega^*_T(Y)$ can be computed explicitly:
$$\Omega^*_T(X)\simeq \bL[[x, t_1, t_2]]/((x^2-t_1^2t_2^2)^2);\quad
\Omega^*_T(Y)\simeq \bL[[x, t_1, t_2]]/((x-t_1t_2)^2).$$
The Weyl group $W_K\simeq\Z/2\Z$ acts by $x\mapsto -x$, $t_i\mapsto -t_i$ for $i=1$,$2$. It is easy to check directly
that $\Omega^*_T(X)^{W_K}\simeq\Omega^*_T(Y)^{W_K}$.
\end{exm}

\end{document}